\documentclass[conference,a4paper]{IEEEtran}
\usepackage{verbatim}
\usepackage{amsmath}
\usepackage{amssymb}
\usepackage{amsfonts}
\usepackage{latexsym}
\usepackage{amscd}
\usepackage{amsthm}
\usepackage{graphicx}
\usepackage{xcolor}
\usepackage{url}
\usepackage{pb-diagram}
\newcommand{\Ac}{\mathcal{A}}
\newcommand{\Nc}{\mathcal{N}}
\newcommand{\FF}{\mathbb{F}}

\newtheorem{example}{Example}
\newtheorem{defn}{Definition}
\newtheorem{lem}{Lemma}
\newtheorem{cor}{Corollary}
\newtheorem{thm}{Theorem}
\newtheorem{pro}{Proposition}

\begin{document}

\sloppy

%% Paper Title
%% You can use linebreaks \\ within to get better formatting as
%% desired.
\title{Explicit Constructions of Quasi-Uniform Codes from Groups}

%% Author names and affiliations:
%%
%% Avoiding spaces at the end of the author lines is not a problem with
%% conference papers because we don't use \thanks or \IEEEmembership.
%%
%% For several authors with only one affiliation:

\author{
   \IEEEauthorblockN{Eldho K. Thomas and Fr\'ed\'erique Oggier}
   \IEEEauthorblockA{Division of Mathematical Sciences\\
     School of Physical and Mathematical Sciences \\
     Nanyang Technological University\\
     Singapore\\
     Email:eldho1@e.ntu.edu.sg,frederique@ntu.edu.sg}
 }

%% To balance the two columns, you should reduce the text-height of
%% the last page using the following command:
%%%%%%%%%%%%%%%%%%%%%%%%%%%%%%%%%%%%%%%%%%%%%%%%%%%%%%%%%%%%%%%%%%%%%
%\addtolength{\textheight}{-9.35cm}
%%%%%%%%%%%%%%%%%%%%%%%%%%%%%%%%%%%%%%%%%%%%%%%%%%%%%%%%%%%%%%%%%%%%%
%% with an appropriate value. This command must be place on the second
%% last page, i.e., for a one-page abstract here, for a two-page
%% abstract right after the \maketitle command.

%% Create the title:
\maketitle

%% Abstract:
%% For the final version of the accepted paper, please make sure you
%% remove the comment "THIS PAPER IS ELIGIBLE FOR THE STUDENT PAPER
%% AWARD."
%%
\begin{abstract}
%\it{``THIS PAPER IS ELIGIBLE FOR THE STUDENT PAPER AWARD''}
 We address the question of constructing explicitly quasi-uniform codes from groups. We determine the size of the codebook, 
the alphabet and the minimum distance as a function of the corresponding group, both for abelian and some nonabelian groups. Potentials applications comprise the design of almost affine codes and non-linear network codes.
\end{abstract}

%*************************************************************************%
%
% INTRO
%
%*************************************************************************%

\section{Introduction}

Let $X_1, \ldots, X_n$ be a collection of $n$ jointly distributed discrete random variables over some alphabet of size $N$.
We denote by $\Ac$ a subset of indices from $\Nc = \{1,\ldots ,n\}$, and $X_\Ac=\{X_i,~i\in\Ac\}$. We call the support of $X_\Ac$ $\lambda(X_\Ac)=\{x_\Ac: Pr(X_\Ac=x_\Ac)>0\}$.
\begin{defn}
A probability distribution over a set of $n$ random variables $X_1,\ldots,X_n$ is said to be {\it quasi-uniform} if for any $\Ac\subseteq\Nc$, $X_\Ac$ is uniformly distributed over its support $\lambda(X_\Ac)$:
\[
P(X_\Ac=x_\Ac)=
\left\{
\begin{array}{ll}
1/|\lambda(X_\Ac)| & \mbox{if }x_\Ac\in\lambda(X_\Ac), \\
0 & \mbox{otherwise}.
\end{array}
\right.
\]
\end{defn}
The motivation for introducing quasi-uniform random variables~\cite{C} is that they possess a non-asymptotic equipartition property, by analogy to the asymptotic equipartition property, where long typical sequences have total probability close to 1, and are approximately uniformly distributed. 

We call a {\it code $C$ of length $n$} an arbitrary nonempty subset of $\mathcal{X}_1\times \cdots \times \mathcal{X}_n$ where $\mathcal{X}_i$ is the alphabet for the $i$th codeword symbol, and each $\mathcal{X}_i$ might be different.

We can associate to every code $C$ a set of random variables~\cite{CGB} by treating each codeword $(X_1,\ldots,X_n)\in C$ as a random vector with probability
\[
P(X_\Nc=x_\Nc)=
\left\{
\begin{array}{ll}
1/|C| & \mbox{if }x_\Nc\in C, \\
0 & \mbox{otherwise}.
\end{array}
\right.
\]
To the $i$th codeword symbol then corresponds a {\it codeword symbol random variable} $X_i$ induced by $C$.

\begin{defn}~\cite{CGB}
A code $C$ is said to be quasi-uniform if
the induced codeword symbol random variables are quasi-uniform.
\end{defn}
Given a code, we explained above how to associate a set of random variables, which might or not end up being quasi-uniform. Conversely, given a set of quasi-uniform random variables, a quasi-uniform code is obtained as follows.
Let $X_1, \ldots, X_n$ be a set of quasi-uniform random variables with probabilities $Pr(X_\Ac=x_\Ac)= 1/|\lambda(X_\Ac)|$ for all $\Ac \subseteq \Nc$. The corresponding quasi-uniform code $C$ of length $n$ is  given by $C = \lambda(X_\Nc)=\{x_\Nc= Pr(X_\Nc=x_\Nc)>0\}$.
Quasi-uniform codes were defined in~\cite{CGB}, where some of their properties were discussed, and importantly their weight enumerator polynomial was computed.

For a linear $(n,k)$ code $C$ of dimension $k$ and length $n$ (over some finite field), the weight enumerator polynomial of $C$ is defined as
$$W_C(x,y)= \sum_{c\in C}x^{n-wt(c)}y^{wt(c)},$$
where $wt(c)$ is the weight of $c$, that is the number of non-zero coefficients of $c$.
For arbitrary codes, rather than the weight of the codewords, the distance between two codewords is of interest. Let $A_r(c)=|\{c' \in C,~|\{ j \in\Nc,~c_j \neq c'_j\}|=r\}|$ be the distance profile of $C$ centered at $c$. Note that we avoid defining $A_r$ using $wt(c-c')$, since this already assumes that the difference of two codewords makes sense. It was shown in~\cite{CGB} that quasi-uniform codes are distance-invariant, meaning that the distance profile does not depend on the choice of $c$. 
\begin{thm}~\cite{CGB}
Let $C$ be a quasi-uniform code of length $n$. Then its weight enumerator $W_C(x,y)=\sum_{j=0}^nA_jx^{n-j}y^j$ is given by
\begin{equation}\label{thm:wt}
W_C(x,y)=\sum_{\Ac\subseteq\Nc}q^{H(X_{\Nc})-H(X_\Ac)}(x-y)^{|\Ac|}y^{n-|\Ac|},
\end{equation}
where $H(X_\Ac)= \log_q(|\lambda(X_\Ac)|)$ is the joint entropy of the induced codeword symbol  quasi-uniform random variables.
\end{thm}
The formula for the weight enumerator shows that it only depends $H(X_\Ac)$. In fact,~\cite{CGB} which introduced quasi-uniform codes focused on their information theoretic properties, rather than on their coding properties. The goal of this paper is to address the construction and understanding of such quasi-uniform codes from a constructive point of view. We will use the group theoretic approach proposed in \cite{chan,T} for constructing quasi-uniform random variables from finite groups. 

More precisely, in Section \ref{sec:group}, we recall the construction of quasi-uniform codes from groups, and compute the size of the corresponding code as a function of the group $G$ we started with. 
We then consider abelian groups in Section \ref{sec:ab}, and compute the alphabet as well as the minimum distance as a function of $G$.
We next move to nonabelian groups in Section \ref{sec:nonab}. The structure of group, even though it is a nonabelian one, allows in some cases to mimic a definition for the minimum distance of the code. 
Potential applications to the design of almost affine codes is mentioned in Section \ref{sec:aa}.

Looking at constructions of codes from groups is motivated by the need to design non-linear codes for network coding (see \cite{CG} for applications of quasi-uniform codes to network coding), apart from designing almost affine codes as mentioned above. 

%*************************************************************************************************%
%
% QUASI UNIFORM CODES FROM GROUPS
%
%************************************************************************************************%

\section{Quasi-Uniform Codes from Groups}
\label{sec:group}

Let $G$ be a finite group of order $|G|$ with $n$ subgroups $G_1,\ldots,G_n$, and $G_\Ac=\cap_{i\in\Ac}G_i$. Given a subgroup $G_i$ of $G$, the (left) coset of $G_i$ in $G$ is defined by $gG_i=\{gh,~h\in G_i\}$.  
The number of (left) cosets of $G_i$ in $G$ is called the index of $G_i$ in $G$ and is denoted by
$[G:G_i]$. It is known from Lagrange Theorem that $[G:G_i]=|G|/|G_i|$.
If $G_i$ is normal, the sets of cosets $G/G_i:=\{gG_i,~g\in G\}$ are themselves groups, called quotient groups.

Let $X$ be a random variable uniformly distributed over $G$, that is $P(X=g)=1/|G|$, for any $g\in G$.
Define the new random variable $X_i=XG_i$, with support the $[G:G_i]$ cosets of $G_i$ in $G$. Then
$P(X_i=gG_i)=|G_i|/|G|$ and $P(X_i=gG_i,~i\in\Ac)=|\cap_{i\in\Ac}G_i|/|G|$.

This shows that quasi-uniform random variables may be obtained from finite groups. More precisely:
\begin{thm}\cite{chan,T}
\label{thm:one}
For any finite group $G$ and any subgroups $G_1, \ldots, G_n$ of $G$, there exist $n$ jointly distributed quasi-uniform discrete random variables $X_1,\ldots, X_n$ such that for all non-empty subsets $\Ac$ of $\Nc$, $Pr(X_\Ac=x_\Ac) = |G_\Ac| / |G |$. 
\end{thm}

Quasi-uniform codes are obtained from these quasi-uniform distributions by taking the support $\lambda(X_\Nc)$, 
as explained in the introduction. Codewords (of length $n$) can then be described explicitly by letting the random variable $X$ take every possible values in the group $G$, and by computing the corresponding cosets as follows:
\begin{center}
\begin{tabular}{c|c|c|c|}
           & $G_1$        & $\hdots$ &  $G_n$ \\
\hline
$g_1$      & $g_1G_1$     &          & $g_1G_n$ \\
$g_2$      &  $g_2G_1$    &          & $g_2G_n$ \\
$\vdots$   &   $\vdots$        &          &  $\vdots$       \\
$g_{|G|}$  & $g_{|G|}G_1$ &$\hdots$ & $g_{|G|}G_n$ \\
\hline
\end{tabular}
\end{center}
Each row corresponds to one codeword of length $n$.
The cardinality $|C|$ of the code obtained seems to be $|G|$, but in fact, it depends on the subgroups $G_1,\ldots,G_n$. Indeed, it could be that the above table yields several copies of the same code.

\begin{lem}\label{lem:sizeC}
Let $C$ be a quasi-uniform code obtained from a group $G$ and subgroups $G_1,\ldots,G_n$. 
Then $|C|=|G|/|G_\Nc|$. In particular, if $|G_\Nc|=1$, then $|C|=|G|$.
\end{lem}
\begin{proof}
Let $G_\Nc=\{h_1,\ldots,h_m\}$ be the intersection of all the subgroups $G_1,\ldots,G_n$. There 
are $|G|/|G_\Nc|$ cosets of $G_\Nc$ in $G$. Let us compute a first coset, say $g_1G_\Nc=\{g_1,g_1h_2,\ldots,g_1h_m\}$, by assuming wlog that $h_1$ is the identity element of $G_\Nc$. In words, we observe that every element in $g_1G_\Nc$ is a multiple of a non-trivial element of $G_\Nc$. Thus, when computing the above table, we have (by reordering the elements of $G$ so as to list first the elements in $g_1G_\Nc$):
\begin{center}
\begin{tabular}{c|c|c|c|}
          & $G_1$              & $\hdots$ & $G_n$ \\
\hline
$g_1$     & $g_1G_1$           &          & $g_1G_n$ \\
$g_1h_2$  & $g_1h_2G_1=g_1G_1$ &          & $g_1h_2G_n=g_1G_n$ \\
$\vdots$  &   $\vdots$         &          &  $\vdots$       \\
$g_1h_m$  & $g_1h_mG_1=g_1G_1$ &$\hdots$  & $g_1h_mG_n=g_1G_n$ \\
 $\vdots$ &   $\vdots$         &          & $\vdots$  \\           
\hline
\end{tabular}
\end{center}
where $g_1h_iG_j=g_1G_j$, for all $i$, $j=1,\ldots,n$ because by definition of $G_\Nc$, 
$h_i\in G_j$. Since the cosets of $G_\Nc$ partition $G$, we will get $|G|/|G_\Nc$ copies of a code $C$.
\end{proof}

One of the motivations to consider quasi-uniform codes is that they allow to go beyond abelian structures. Nevertheless, we will start by considering the case of abelian groups, which is easier to handle.

%*************************************************************************************************%
%
% abelian
% 
%************************************************************%
\section{Quasi-Uniform Codes from Abelian Groups}
\label{sec:ab}

Suppose that $G$ is an abelian group, with subgroups $G_1,\ldots,G_n$. The procedure from Section \ref{sec:group} explains how to obtain a quasi-uniform distribution of $n$ random variables, and thus a quasi-uniform code of length $n$ from $G$. To avoid getting several copies of the same code, as exposed in Lemma \ref{lem:sizeC}, notice that since $G$ is abelian, all the subgroups $G_1,\ldots,G_n$ are normal, and thus so is $G_\Nc$. If $|G_\Nc|>1$, we consider instead of $G$ the quotient group $G/G_\Nc$, and we can thus assume wlog that $|G_\Nc|=1$.

\begin{lem}
The size of the code alphabet is $\sum_{i=1}^n[G:G_i]$.
\end{lem}
\begin{IEEEproof}
It is enough to show that all the cosets that appear in the table are distinct.
By definition, every column contains $[G:G_i]$ distinct cosets. If $|G_i|\neq |G_j|$, the respective cosets 
will have different sizes, so let us assume that $|G_i|=|G_j|$. 
If $gG_i=g'G_j$, then $g^{-1}gG_i= G_i = g^{-1}g'G_j$ and it must be that $g^{-1}g'G_j $ is a subgroup. 
This implies that $g^{-1}g' \in G_j$, and thus that $g^{-1}g'G_j=G_j$.
\end{IEEEproof}
The size of the alphabet can often be reduced, as explained next.
Let $\pi_i$ denote the canonical projection $\pi_i:G\rightarrow G/G_i$. 
Since $G/G_i$ is itself an abelian group, let us denote explicitly by $\psi_i: G/G_i \rightarrow H_i$ this group isomorphism. Then $\pi_i(g)=gG_i \mapsto h \in H_i$ via $\psi_i(gG_i)=h$, $i=1,\ldots,n$.

\begin{pro}
\label{lem:two}
Let $G$ be an abelian group with subgroups $G_1,\ldots,G_n$. Then its corresponding quasi-uniform code is defined
over $H_1 \times \cdots \times H_n$. 
\end{pro}
\begin{IEEEproof}
Let $X$ be again this random variable defined over $G$ by $Pr(X=g)=1/|G|$.
Define a new random variable $Z_i$ by $Z_i=\psi_i(\pi_i(X))$ which takes values directly in $H_i$.
Then $Pr(Z_i=h)=Pr(\psi_i(\pi_i(X))=h)=Pr(\pi_i(X)=gG_i)=|G_i|/|G|.$
Similarly, if $Pr(Z_i=h_i,i\in \Ac) >0$, then
$Pr(Z_i=h_i:i\in \Ac) = Pr(\psi_i(\pi_i(X))=h_i:i \in \Ac) = Pr( \pi_i(X) = gG_i: i \in \Ac)  =  |G_\Ac|/|G|.$
\end{IEEEproof}

In other words, we get a labeling of the cosets which respects the group structures componentwise. The next result then follows naturally.

\begin{cor}
\label{cor:abelian}
A quasi-uniform code $C$ obtained from an abelian group is itself an abelian group.
\end{cor}
\begin{IEEEproof}
First notice that the zero codeword is in $C$, since the codeword corresponding to the identity element in $G$ is 
$(\psi_1(\pi_1(G_1)),\ldots,\psi_n(\pi_n(G_n)))=(0,\ldots,0)$ where each $0$ corresponds to the identity element in each abelian group $H_i$.
Let $(\psi_1(\pi_1(gG_1)),\ldots,\psi_n(\pi_n(gG_n)))$ and $(\psi_1(\pi_1(g'G_1)),\ldots,\psi_n(\pi_n(g'G_n)))$ be two codewords in $C$. Then note that the codeword in $C$ corresponding to the element $g+g' \in G$ is 
$(\psi_1(\pi_1((g+g')G_1)),\ldots,\psi_n(\pi_n((g+g')G_n)))$ where $\psi_i(\pi_i(g+g'))=\psi_i(\pi_i(g))+\psi_i(\pi_i(g'))$, $i=1,\ldots,n$. Every codeword has an additive inverse for the same reason. It forms an abelian group because every group law componentwise is commutative.
\end{IEEEproof}

Because every $H_i$ is an abelian group, we can freely use $0$ since it corresponds to the identity element of $H_i$, as well as the operation $+$ and $-$, since $+$ is the group law for $H_i$, and $-$ is the additive inverse. However, the alphabet $H_i$ is possibly any abelian group, in particular, different abelian groups might be used for different components of the codewords. The classification of abelian groups tells us that each $H_i$ can
be expressed as the direct sum of cyclic subgroups of order a prime power. In the particular case where we have only one cyclic group, then (1) the group $C_{p^r}$ is isomorphic to the integers mod $p^r$, and (2) the group $C_p$ is isomorphic to the integers mod $p$, which in fact has a field structure, and we deal with the usual finite field $\FF_p$. If all the subgroups $G_1,\ldots,G_n$ have index $p$, then we get an $(n,k)$ linear code over $\FF_p$.

The minimum distance of an abelian quasi-uniform code is encoded in its weight enumerator, but is however not easily read from (\ref{thm:wt}). We can easily express it in terms of the subgroups $G_1,\ldots,G_n$.

\begin{lem}\label{lem:dmin}
The minimum distance $\min_{c\in C}wt(c)$ of a quasi-uniform code $C$ generated by an abelian group $G$ and its subgroups $G_i$, $i=1,\ldots n$ is $n-\max_{\Ac\in \Nc,G_\Ac \neq \{0\}}|\Ac|$.
\end{lem}
\begin{IEEEproof}
The minimum distance $\min_{c\neq c' \in C}|\{c_i \neq c'_i\}|$ can be written as $\min_{c\neq c' \in C}wt(c-c')$ since $c-c'$ makes sense. Furthermore, since $c-c'\in C$, it reduces, as for linear codes over finite fields, 
to $\min_{c\in C}wt(c)$, and we are left to find the weight of the codeword having maximum number of zeros. 
This corresponds to finding the maximum number of subgroups whose intersection contains a non-trivial element of $G$. 
\end{IEEEproof}

As an illustration, here is a simple family of abelian groups that generate $(n,k)$ linear codes over $\FF_p$.
\begin{lem}
The elementary abelian group $C_p\times C_p$ generates a $(p+1,2)$ linear code over $\FF_p$ with minimum distance $p$. 
\end{lem}
\begin{IEEEproof}
The group $G=C_p\times C_p$ contains $p+1$ non-trivial subgroups, of the form $\langle (1,i)\rangle$ where $i=0,1,\ldots p-1$ and $\langle (0,1) \rangle$. They all have index $p$ and trivial pairwise intersection. We thus get a code of length $n=p+1$, containing $p^2$ codewords, which is linear over $\FF_p$ (by using that $C_p$ is isomorphic to the integers mod $p$).
Since the pairwise intersection of subgroups is trivial, the minimum distance is $p$. 
\end{IEEEproof}

We finish this section by providing a worked out example.

\begin{example}\rm
\label{exm:almost}
Consider the elementary abelian group $G=C_3 \times C_3 \simeq \{0,1,2\} \times \{0,1,2\}$ and the
four subgroups $G_1=\langle (1,0) \rangle = \{ (0,0), (1,0), (2,0) \}$, $G_2=\langle (0,1) \rangle = \{ (0,0), (0,1), (0,2) \}$, $G_3=\langle (1,1) \rangle = \{ (0,0), (1,1), (2,2) \}$, and  $G_4=\langle (1,2) \rangle = \{ (0,0), (1,2), (2,1) \}$.
Using the method of Section \ref{sec:group}, we obtain the following codewords (we write $ij$ instead of $(i,j)$ for brevity):
\begin{center}
\small{
\begin{tabular}{c|c|c|c|c|}
     & $\langle (10) \rangle$ & $\langle (01) \rangle$ & $\langle (11) \rangle$ & $\langle (12) \rangle$ \\
\hline
$\!\!\!(00)\!$ & $\langle (10) \rangle$ &  $\langle (01) \rangle$ & $\langle (11) \rangle$ & $\langle (12) \rangle$ \\
$\!\!\!(01)\!$ & $\!(01)(11)(21)\!$   &  $\langle (01) \rangle$ & $\!(01)(12)(20)\!$   & $\!(01)(10)(22)\!$  \\
$\!\!(02)\!$ & $\!(02)(12)(22)\!$   &  $\langle (01) \rangle$ & $\!(10)(21)(02)\!$   & $\!(02)(11)(20)\!$  \\
$\!\!\!(10)\!$ & $\langle (10) \rangle$  & $\!(10)(11)(12)\!$ & $\!(10)(21)(02)\!$   & $\!(01)(10)(22)\!$  \\
$\!\!\!(11)\!$ & $\!(01)(11)(21)\!$   & $\!(10)(11)(12)\!$  &$\langle (11) \rangle$ & $\!(02)(11)(20)\!$  \\
$\!\!\!(12)\!$ & $\!(02)(12)(22)\!$   &  $\!(10)(11)(12)\!$ & $\!(01)(12)(20)\!$  & $\langle (12) \rangle$ \\
$\!\!\!(20)\!$ &$\langle (10) \rangle$ &$\!(20)(21)(22)\!$ &$\!(01)(12)(20)\!$  & $\!(02)(11)(20)\!$ \\
$\!\!\!(21)\!$ &$\!(01)(11)(21)\!$ &$\!(20)(21)(22)\!$ &$\!(10)(21)(02)\!$ & $\langle (12) \rangle$ \\
$\!\!\!(22)\!$ &$\!(02)(12)(22)\!$ & $\!(20)(21)(22)\!$&$\langle (11) \rangle$  & $\!(01)(10)(22)\!$  \\
\end{tabular}}
\end{center}
Now let $H_1=G/\langle (10) \rangle, ~H_2 =G/\langle (01) \rangle,~ H_3= G/\langle (11) \rangle$ and $H_4= G/\langle (12) \rangle$. Note that $H_i \simeq C_3=\{0,1,2\}$ for all $i$.
If we replace the subgroups by their quotients in the above table, we get the following code from Lemma~\ref{lem:two}:

\begin{center}
\begin{tabular}{c|c|c|c|c|}
     & $\langle (10) \rangle$ & $\langle (01) \rangle$ & $\langle (11) \rangle$ & $\langle (12) \rangle$ \\
\hline
$\!\!\!(00)\!$ & $0$ & $0  $ & $0    $ & $0  $ \\
$\!\!\!(01)\!$ & $1$ & $0  $ & $1$ & $1$ \\
$\!\!(02)\!$   & $2$ & $0  $ & $2$ & $2$ \\
$\!\!\!(10)\!$ & $0$ & $1  $ & $2$ & $1$ \\
$\!\!\!(11)\!$ & $1$ & $1  $ & $0   $ & $2$ \\
$\!\!\!(12)\!$ & $2$ & $1 $ & $1$ & $0      $ \\
$\!\!\!(20)\!$ & $0$ & $2 $ & $1$ & $2$ \\
$\!\!\!(21)\!$ & $1$ & $2 $ & $2$ & $0   $ \\
$\!\!\!(22)\!$ & $2$ & $2 $ & $0 $ & $1$ \\
\end{tabular}
\end{center}
It is a ternary linear code of length $p+1=4$ and minimum distance $p=3$ with generator matrix
\[ \left( \begin{array}{cccc}
1 & 0 & 1 & 1\\
0 & 1 & 2 &1
\end{array} \right).
\]
Since $H(X_\Nc)=\log_q 9$, $H(X_\Ac)=\log_q 3$ when $|\Ac|=1$ and $H(X_\Ac)=\log_q 9$ if $|\Ac|\geq 2$,   
we have that $q^{H(X_\Nc)-H(X_\Ac)}=9$ when $\Ac$ is empty, $q^{H(X_\Nc)-H(X_\Ac)}=q^{\log_q(9/3)}=3$ when $|\Ac|=1$ and $1$ otherwise, so that 
the weight enumerator of $C$ is, using Theorem \ref{thm:wt}:
\begin{eqnarray*}
W_C(x,y) &=& 9 y^4 + {4 \choose 1}3 (x-y)y^3 \\
         &  & +{4 \choose 2}(x-y)^2y^2+ {4 \choose 3}(x-y)^3y+(x-y)^4\\
         &=& x^4 + 8xy^3,
\end{eqnarray*}
as is clearly the case.
\end{example}

%*************************************************************************************************%
%
% NON abelian
%
%*************************************************************************************************%
\section{Quasi-Uniform Codes from Nonabelian Groups}
\label{sec:nonab}

Suppose now that $G$ is a nonabelian group. The resulting quasi-uniform codes might end up 
being very different depending on the nature of the subgroups considered. We next treat the different possible cases that can occur.

%*********************************************************************************%
\subsection{The Case of Quotient Groups}

If $G$ is a nonabelian group, but $G_1,\ldots,G_n$ are normal subgroups, then the intersection $G_\Nc$ of all subgroups $G_1,\ldots,G_n$ is a normal subgroup. Following Lemma \ref{lem:sizeC}, we can further consider the quotient $G/G_\Nc$. 
As a result, some nonabelian groups are really reduced to abelian ones. This is the case of some dihedral groups. Let 
\[
D_{2m}=\langle r,s~|~r^m=s^2=1,~rs=sr^{-1} \rangle
\]
be the dihedral group of order $2m$.
\begin{lem}\label{lem:dih}
Quasi-uniform codes obtained from dihedral groups whose order is a power of 2 and some (possiby all) of their normal subgroups are obtained from abelian groups.
\end{lem}
\begin{IEEEproof}
Normal subgroups $H$ of $D_{2m}$ are known: $H$ is either a subgroup of $\langle r \rangle$, or $2|m$ and $H$ is one of the two maximal subgroups of index 2 $\langle r^2,s\rangle$, $\langle r^2,rs\rangle$. Since $m$ is a power of 2, the intersection $G_\Nc$ of (any choice of) these normal subgroups is necessarily a subgroup of order some power of 2, and we can take the quotient $D_{2m}/G_\Nc$, which is a dihedral group of smaller order also a power of 2. But then its normal subgroups will be of the same form, and the same process can be iterated, until we reach $D_4$ which is abelian.
\end{IEEEproof}

The case when $G$ is nonabelian and all the subgroups $G_1,\ldots,G_n$ are normal but some (possibly all) of the quotient groups $G/G_i$ are nonabelian is very interesting. 
Indeed, in that case, Proposition \ref{lem:two} still holds (with $H_i$ nonabelian), and in fact, the corresponding quasi-uniform code still has a group structure, but that of an nonabelian group.
This gives the opportunity to go beyond abelian structures, and yet to keep a group structure.
As recalled above, we may assume wlog that $|G_\Nc|=1$.

\begin{lem}
A quasi-uniform code C obtained from a nonabelian group $G$ with normal subgroups $G_1,\ldots,G_n$ where at least one quotient group $G/G_i$ is nonabelian forms a nonabelian group.
\end{lem} 
The proof is identical to that of Corollary \ref{cor:abelian}, but one has to be cautious that the group law is not commutative. The group law $*$ is defined componentwise, and the identity element is the codeword $(1_{H_1},\ldots,1_{H_n})$.

It is then possible to mimic the definition of minimum distance. The weight of a codeword is then the number of components which are not an identity element. The identity codeword plays the role of the whole zero codeword. The minimum distance $\min_{c\neq c'\in C}|\{c_i \neq c'_i\}|$ is 
then $\min_{c\neq c'\in C}wt(c*(c')^{-1})$. Indeed, $(c')^{-1}=((c'_1)^{-1},\ldots,(c'_n)^{-1})$, and noncommutativity is not an issue, since every inverse componentwise is both a left and a right inverse. This gives a counterpart to Lemma \ref{lem:dmin}.

\begin{cor}\label{cor:dmin}
The minimum distance $\min_{c\in C}wt(c)$ of a quasi-uniform code $C$ generated by a nonabelian group $G$ and its normal subgroups $G_i$, $i=1,\ldots n$ is $n-\max_{\Ac\in \Nc,G_\Ac \neq \{0\}}|\Ac|$, where the weight $wt(c)$ is understood as the number of components which are not an identity element in some $H_i$.
\end{cor}

\begin{figure}
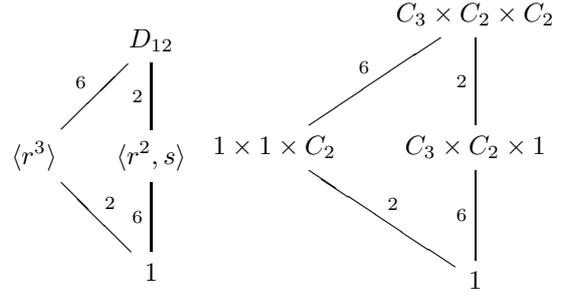

\[
\begin{diagram}
\node{} \node{D_{12}} \\
\node{\langle r^3 \rangle} \arrow{ne,l,-}{6} \node{\langle r^2,s \rangle}\arrow{n,l,-}{2} \\
\node{} \node{1}\arrow{nw,l,-}{2}\arrow{n,l,-}{6}\\
\end{diagram} 
\begin{diagram}
\node{} \node{C_3\times C_2\times C_2} \\
 \node{1\times 1 \times C_2} \arrow{ne,l,-}{6}\node{C_3\times C_2 \times 1}\arrow{n,l,-}{2} \\
\node{} \node{1}\arrow{nw,l,-}{2}\arrow{n,l,-}{6}\\
\end{diagram}
\]
\vspace{-1.7cm}
\caption{\label{fig:d12}
On the right, the dihedral group $D_{12}$, and on the left, the abelian group $C_3\times C_2 \times C_2$, both with some of their subgroups. 
}
\end{figure}

\begin{example}\rm
The dihedral group $D_{12}$ has two normal subgroups $G_1=\langle r^3 \rangle$ and $G_2=\langle r^2,s\rangle$ with trivial intersection (see Figure \ref{fig:d12}). We can create quasi-uniform codes of length $n$ by choosing $n-2$ other normal subgroups. Since $D_{12}/G_1 \simeq D_6$, this gives a nonabelian quotient $H_1$.
\end{example}

This example illustrates the difference between an information theoretic view of these codes, where one focuses on the joint entropy of the corresponding quasi-uniform codes, and a coding perspective, where the actual code, its alphabet, and its structure are of interest. We observe on Figure 
\ref{fig:d12} that if we care about having two (normal) subgroups $G_1,G_2$ with respective order 2 and 6, in a group of order 12, we could have done that with the abelian group $C_3\times C_2 \times C_2$. The difference is in the alphabet: $D_{12}/G_1$ is a nonabelian group, while $C_3\times C_2\times C_2/1\times 1\times C_2$ is abelian. This question of distinguishing nonabelian groups whose entropic vectors can or cannot be obtained from abelian groups was addressed more generally in \cite{ENF}.

\begin{defn}\cite{ENF}
Let $G$ be a nonabelian group and let $G_1, \ldots, G_n$ be fixed subgroups of $G$.
Suppose there exists an abelian group $A$  with subgroups $A_1, \ldots, A_n$ such that for every non-empty $\Ac \subseteq\Nc$, $[G:G_\Ac] = [A: A_\Ac]$. Then we say that
$(A, A_1, \dots, A_n)$ {\it{represents}} $(G, G_1, \ldots, G_n)$.
\end{defn}

It follows immediately that if $(A,A_1,\ldots,A_n)$ represents $(G,G_1,\ldots,G_n)$, then the quasi-uniform codes generated by both groups have the same joint entropy, and in turn the same weight enumerator, by Theorem \ref{thm:wt}. This is the case for dihedral and quasi-dihedral 2-groups as well as dicyclic 2-groups, which are abelian group representable for any number of subgroups, and all nilpotent groups for $n=2$~\cite{ENF}. 

If $(G,G_1,\ldots,G_n)$ cannot be abelian represented, then we can build a quasi-uniform code using exactly the same subgroups $G_1,\ldots,G_n$ and be sure that its weight enumerator cannot be obtained from an abelian group.
It was also shown in \cite{ENF} that dihedral groups whose order is a power of 2 are abelian group representable, however Lemma \ref{lem:dih} is stronger, in that it shows that the code alphabet will be the same. 
%*************************************************************************************************%
\subsection{The Case of Nonnormal Subgroups}

Let $G$ be a nonabelian group, with subgroups $G_1,\ldots,G_n$, where some (possibly all) of the subgroups are not normal. In that case, we lose the group structure on the set of cosets.
To start with, Lemma \ref{lem:sizeC} will still hold, since it does not depend on the normality of the subgroups, however in general, we cannot take the quotient by $G_\Nc$ anymore, and the copy of $C$ we are left with does not have a group structure. This is an example of situations mentioned in the introduction, where the minimum distance is then not of interest.

\begin{example}\rm
Consider the group $S_3=\{ (),(12),(13),(23),(123),(132) \}$ of permutations on three elements described in cycle notation.
The corresponding quasi-uniform code is: 
\begin{center}
\begin{tabular}{c|c|c|c|}
     & $\langle (12) \rangle$ & $\langle (13) \rangle$ & $\langle (23) \rangle$ \\
\hline
$()$   & $\langle (12) \rangle$ & $\langle (13) \rangle$ & $\langle (23) \rangle$ \\
$(12)$ & $\langle (12) \rangle$ & $(12),(132)$ & $(12),(123)$ \\
$(13)$ & $(13),(123)$& $\langle (13) \rangle$& $(13),(132)$ \\
$(23)$ &$(23),(132)$ & $(23),(123)$ & $\langle (23) \rangle$ \\
$(123)$& $(13),(123) $ &$(23),(123) $ &$(12),(123) $ \\
$(132)$&$(23),(132) $ & $ (12),(132) $& $(13),(132)$ \\
\end{tabular}
\end{center}
Note that we could label the cosets using integers, however then one should keep in mind that 
these integers do not have any algebraic meaning.
\end{example}

Losing the group structure of the code however has the advantage that we have more flexibility in choosing the subgroups we deal with, and thus have more choices in terms of possible intersections that we are getting. 

\begin{pro}
There exist a quasi-uniform code obtained from a non-nilpotent group which cannot be obtained by any abelian group. 
\end{pro}
\begin{IEEEproof}
Let $G$ be a non-nilpotent group. Then $G$ is not abelian group representable for all $n$~\cite{ENF}. That is, there exists some subgroups $G_i$ such that $[G:G_\Ac] \ne [A:A_\Ac]$ for some abelian group and subgroups. The proof uses the fact that in a group which is not nilpotent, there exists a Sylow subgroup which is not normal.
\end{IEEEproof}

%*************************************************************************%
%
% ALMOST AFFINE
%
%*************************************************************************%

\section{Almost Affine Codes from Groups}
\label{sec:aa}

Consider a non-empty subset $\Ac$ of $\Nc$ and let $C_\Ac$ be the projection of the code $C$ into the coordinate space $\Ac$, that is, all the words of $C$ are restricted to the positions in $\Ac$.
\begin{defn}~\cite{SA}
A $q$-ary code $C$ of length $n$ is said to be {\it almost affine} if it satisfies the condition  
\[
\log_q(|C_\Ac|) \in \mathbb{N},~\mbox {for all }\Ac \subseteq \Nc.
\] 
\end{defn}
Almost affine codes were introduced in \cite{SA} as a generalization of affine codes, which are themselves generalizations of linear codes over finite fields.
It was shown in \cite{CGB} that almost affine codes are quasi-uniform. It is thus natural to look for such codes among codes built from groups. Because of the definition of almost affine codes, $p$-groups are the first candidates that come to mind.

\begin{lem}
Let $G$ be a $p$-group, and let $G_1,\ldots,G_n$ be subgroups of index $p$. 
The corresponding quasi-uniform code is almost affine.
\end{lem}
\begin{IEEEproof}
First note that $G_1,\ldots,G_n$ are normal subgroups of $G$, thus so is their intersection, 
and wlog we may assume that $|G_\Nc|=1$. Then Proposition \ref{lem:two} holds (though the $H_i$ might be nonabelian), and we obtain a $p$-ary quasi-uniform code. Since any intersection $G_\Ac$ will have order a power of $p$, the code obtained is almost affine.
\end{IEEEproof}
There are other ways to get almost affine codes from $p$-groups, and
$p$-groups are not the only finite groups that can provide almost affine codes.

%***********************************************%
%
% CONCLUSION
%
%*************************************************%

\section{Conclusion}

Quasi-uniform codes were known to be constructed from groups. 
In this paper, we were interested in relating the properties of the obtained code 
as a function of the corresponding group. We determined the size of the code, 
its alphabet, and its minimum distance, both for abelian groups, but also for some nonabelian 
groups where the group structure allows to mimic the definition of minimum distance. 
An application to the design of almost affine codes is also given. 

Current and future works involve studying further properties of these codes coming from groups, 
and in particular (1) codes coming from nonabelian groups which cannot be reduced to abelian groups, and (2) almost affine codes. The information theoretic point of view is also of course of interest: it is related to the understanding of entropic vectors coming from nonabelian groups. A long term goal is the design of non-linear network codes.

%% Appendix:
%% If needed a single appendix is created by
%\appendix
%% If several appendices are needed, then the command
%\appendices
%% in combination with further \section-commands can be used.

%% Use \section* for acknowledgement
\section*{Acknowledgment}

The work of E. Thomas and F. Oggier is supported by the Nanyang Technological University under Research Grant M58110049. 

%% References:
%% We recommend the usage of BibTeX:
%%
%\bibliographystyle{IEEEtran}
%\bibliography{definitions,bibliofile}
%%
%% where we here have assume the existence of the files
%% definitions.bib and bibliofile.bib.
%% BibTeX documentation can be obtained at:
%% http://www.ctan.org/tex-archive/biblio/bibtex/contrib/doc/
%%
%%
%%
%% Or manual references (pay attention to consistency!):

\end{document}